\documentclass[12pt,a4paper]{article}
\usepackage{indentfirst}
\usepackage{amsmath,amssymb,amsfonts,amsthm,graphics}
\usepackage{makeidx}
\usepackage{ifpdf}
\newtheorem{thm}{Theorem}

\newtheorem{lem}{Lemma}

\theoremstyle{definition}

\def\-{\mbox{--}}

\newtheorem{pro}{Proposition}

\newtheorem{obs}{Observation}

\newtheorem{definition}{Definition}

\begin{document}
\title{Graphs with 3-rainbow index $n-1$ and $n-2$\Large\bf \footnote{Supported by NSFC No.11071130.}}
\author{\small Xueliang~Li, Kang Yang, Yan~Zhao\\
\small Center for Combinatorics and LPMC-TJKLC\\
\small Nankai University, Tianjin 300071, China\\
\small  lxl@nankai.edu.cn;
 yangkang@mail.nankai.edu.cn; zhaoyan2010@mail.nankai.edu.cn}
\date{}
\maketitle
\begin{abstract}

Let $G$ be a nontrivial connected graph with an edge-coloring $c:E(G)\rightarrow \{1,2,\ldots,q\},$
$q\in \mathbb{N}$, where adjacent edges may be colored the same. A tree $T$ in $G$ is a $rainbow~tree$ if no two edges of $T$ receive the same color. For a vertex set $S\subseteq V(G)$, the tree connecting $S$ in $G$ is called an $S$-tree. The minimum number of colors that are needed in an edge-coloring of $G$ such that there is a rainbow $S$-tree for each $k$-set $S$ of $V(G)$ is called the $k$-rainbow index of $G$, denoted by $rx_k(G)$. In \cite{Zhang}, they got that the $k$-rainbow index of a tree is $n-1$ and the $k$-rainbow index of a unicyclic graph is $n-1$ or $n-2$. So there is an intriguing problem: Characterize graphs with the $k$-rainbow index $n-1$ and $n-2$. In this paper, we focus on $k=3$, and characterize the graphs whose 3-rainbow index is $n-1$ and $n-2$, respectively.



\end{abstract}

\section{Introduction}
All graphs considered in this paper are simple, finite and
undirected. We follow the terminology and notation of Bondy and
Murty \cite{Bondy}. Let $G$ be a nontrivial connected graph with an
edge-coloring $c: E(G)\rightarrow \{1,2,\ldots,q\}$, $q\in
\mathbb{N}$, where adjacent edges may be colored the same. A path of
$G$ is a \emph{rainbow path} if every two edges of the path have
distinct colors. The graph $G$ is \emph{rainbow connected} if for
every two vertices $u$ and $v$ of $G$, there is a rainbow path
connecting $u$ and $v$. The minimum number of colors for which there
is an edge coloring of $G$ such that $G$ is rainbow connected is
called the \emph{rainbow connection number}, denoted by $rc(G)$.
Results on the rainbow connectivity can be found in
\cite{Chartrand1,Caro, Chartrand,ChartrandZhang,LSS, LiSun}.

These concepts were introduced by Chartrand et al. in \cite{Chartrand1}. In \cite{Zhang}, they generalized the concept of  rainbow path to rainbow tree.
A tree $T$ in $G$ is a $rainbow~tree$ if no two edges of $T$ receive the same color. For $S\subseteq V (G)$, a $rainbow\ S$-$tree$ is a rainbow tree connecting $S$. Given a fixed integer $k$  with $2\leq k \leq n$, the edge-coloring $c$ of $G$ is called a $k$-$rainbow~coloring$ of $G$ if for every $k$-subset $S$ of $V(G)$, there exists a rainbow $S$-tree. In this case,  $G$ is called $k$-$rainbow~connected$. The minimum number of colors that are needed in a $k$-$rainbow~coloring$ of $G$ is called the $k$-$rainbow~index$ of $G$, denoted by $rx_k(G)$. Clearly, when $k=2$, $rx_2(G)$ is nothing new but the rainbow connection number $rc(G)$ of $G$. For every connected graph $G$ of order $n$, it is easy to see that $rx_2(G)\leq rx_3(G)\leq \cdots \leq rx_n(G)$.

The $Steiner~distance$ $d_G(S)$ of a set $S$ of vertices in $G$ is the minimum size of a tree in $G$ connecting $S$. The $k$-$Steiner~diameter$ $sdiam_k(G)$ of $G$ is the maximum Steiner distance of $S$ among all sets $S$ with $k$ vertices in $G$. Then there is a simple upper bound and lower bound for $rx_k(G)$.

\begin{obs}[\cite{Zhang}]\label{obs1}
For every connected graph $G$ of order $n\geq 3$ and each integer $k$ with $3\leq k\leq n$,
$k-1\leq sdiam_k(G)\leq rx_k(G)\leq n-1$.
\end{obs}

They showed that trees are composed of a class of graphs whose k-rainbow index attains
the upper bound.

\begin{pro}[\cite{Zhang}]\label{pro1}
Let $T$ be a tree of order $n\geq 3$. For each integer $k$ with $3\leq k\leq n$, $rx_k(T)=n-1$.
\end{pro}

They also showed that the $k$-rainbow index of a unicyclic graph is $n-1$ or $n-2$.

\begin{thm}[\cite{Zhang}]\label{thm1}
If $G$ is a unicyclic graph of order $n\geq 3$ and girth $g\geq 3$, then
\begin{equation}
 rx_k(G)=
   \begin{cases}
      n-2, & \text{$k=3$ and $g\geq4$}; \\
     n-1, & \text{$g=3$ or $4\leq k\leq n$}.
    \end{cases}
\end{equation}
\end{thm}

A natural thought is that which graph of order $n$ has the $k$-rainbow index $n-1$ except for a tree and a unicyclic graph of girth 3? Furthermore, which graph of order $n$ has the $k$-rainbow index $n-2$ except for a unicyclic graph of girth at least 4? In this paper, we focus on $k=3$. In addition, some known results are mentioned.

\begin{obs}[\cite{Zhang}]\label{obs2}
Let $G$ be a connected graph of order $n$ containing two bridges $e$ and $f$. For each integer $k$ with $2\leq k\leq n$, every $k$-rainbow coloring of $G$ must assign distinct colors to $e$ and $f$.
\end{obs}

\begin{lem}[\cite{CLYZ}]\label{lem3}
If $G$ is a connected graph and $\{H_1, H_2,\cdots, H_k\}$ is a partition of $V(G)$ into connected subgraphs, then $rx_3(G)\leq k-1+\sum_{i=1}^krx_3(H_i)$.
\end{lem}

\begin{thm}[\cite{CLYZ}]\label{thm3}
Let $G$ be a connected graph of order $n$. Then $rx_3(G)=2$ if and only if $G=K_5$ or $G$ is a 2-connected graph of order 4 or $G$ is of order 3.
\end{thm}

\begin{obs}[\cite{CLYZ}]\label{obs3}
Let $G$ be a connected graph of order $n$, and $H$ be a connected spanning subgraph of $G$. Then $rx_3(G)\leq rx_3(H)$.
\end{obs}

This paper is organized as follows. In section 2, some basic results and notations are presented. In section 3, we characterize the graphs whose 3-rainbow index is $n-1$ and $n-2$, respectively. And we take two steps for the latter case. First, we consider the bicyclic graphs. Second, we consider the tricyclic graphs. Finally, we characterize the graphs whose 3-rainbow index is $n-2$.

\section{Some basic results}
First of all, we need some more terminology and notations.

\begin{definition}
Let $G$ be a connected graph with $n$ vertices and $m$ edges. Define the $cyclomatic~number$ of $G$ as $c(G)=m-n+1$. A graph $G$ with $c(G)=k$ is called a $k$-$cyclic$ graph. According to this definition, if a graph $G$ meets $c(G)=0$, 1, 2 or 3, then the graph $G$ is called acyclic(or a tree), unicyclic, bicyclic, or tricyclic, respectively.
\end{definition}

\begin{definition}
For a subgraph $H$ of $G$ and $v\in V(G)$, let $d(v,H)=min\{d_G(v,x): x\in V(H)\}$.
\end{definition}

Next we define some new notations.

\begin{definition}
For a connected graph $G$ of order $n$,  set $V(G)=\{v_1,v_2,\cdots,v_n\}$, we define a $lexicographic$ $ordering$ between any two edges of $G$ by $v_iv_j<v_sv_t$ if and only if $i<s$ or $i=s$, $j<t$.
\end{definition}

Note that, the lexicographic ordering of a connected graph is unique. Given a coloring $c$ of a connected graph $G$, denote by $c_{\ell}(G)$ a sequence of colors of the edges which are ordered by the lexicographic ordering.

Let $G$ be a connected graph, to $contract$ an edge $e=xy$ is to delete $e$ and replace its ends by a single vertex incident to all the edges which were incident to either $x$ or $y$. Let $G^{'}$ be the graph obtained by contracting some edges of $G$. Given a rainbow coloring of $G^{'}$, when it comes back to $G$, we keep the colors of corresponding edges of $G^{'}$ in $G$ and assign  a new color to a new edge, which makes $G$ 3-rainbow connected. Hence, the following lemma holds.

\begin{lem}\label{lem1}
Let $G$ be a connected graph, and $G^{'}$ be a connected graph by contracting some edges of $G$. Then $rx_3(G)\leq rx_3(G^{'})+|G|-|G^{'}|$.
\end{lem}

\begin{definition}
Let $G_0$ be the graph obtained by contracting all the cut edges of $G$,  then
$G_0$ is called the $basic$ $graph$ of $G$.
\end{definition}


\section{Main results}

\subsection{Characterize the graphs with $rx_3(G)\bf{=n-1}$}
\begin{thm}\label{thm4}
Let $G$ be a connected graph of order $n$. Then $rx_3(G)=n-1$ if and only if $G$ is a tree or $G$ is a unicyclic graph with girth 3.
\end{thm}
\begin{proof}
If $G$ is a tree or a unicyclic graph with girth 3, by Proposition \ref{pro1} and Theorem \ref{thm1}, $rx_3(G)=n-1$. Conversely, suppose $G$ is a graph with $rx_3(G)=n-1$ but not a tree, then $G$ must contain cycles. Let $\{H_1, H_2,\cdots, H_k\}$ be a partition of $V(G)$ into connected subgraphs. If $G$ contains a cycle of length $r$ at least 4, then let $H_1$ be the $r$-cycle, and each other subgraph a single vertex. We color $H_1$ with $r-2$ dedicated colors, then by Lemma \ref{lem3}, $rx_3(G)\leq n-r+rx_3(H_1)=n-2$. Suppose then $G$ contains at least two triangles $C_1$ and $C_2$. If $C_1$ and $C_2$ have a vertex in common, then let $H_1$ be the union of $C_1$ and $C_2$, and each other subgraph a single vertex. We color both $C_1$ and $C_2$ with the same three dedicated colors, thus $rx_3(G)\leq n-5+rx_3(H_1)=n-2$. If $C_1$ and $C_2$ are vertex disjoint, then let $H_1=C_1$, $H_2=C_2$, and each other subgraph a single vertex. We color $H_1$ with three new colors and $H_2$ with the same three colors of $H_1$, thus $rx_3(G)\leq n-5+rx_3(H_1)+rx_3(H_2)=n-2$. Combining the above two cases, $G$ is a unicyclic graph with girth 3. Therefore, the result holds.
\end{proof}

\subsection{Characterize the graphs with $rx_3(G)\bf{=n-2}$}
Next, we characterize the graphs whose 3-rainbow index is $n-2$. We begin with a useful theorem from \cite{CLYZ}.

A 3-$sun$ is a graph constructed from a cycle $C_6=v_1v_2\cdots v_6v_1$ by adding three edges $v_2v_4$, $v_2v_6$ and $v_4v_6$.

\begin{figure}[h,t,b,p]
\begin{center}
\scalebox{1}[1]{\includegraphics{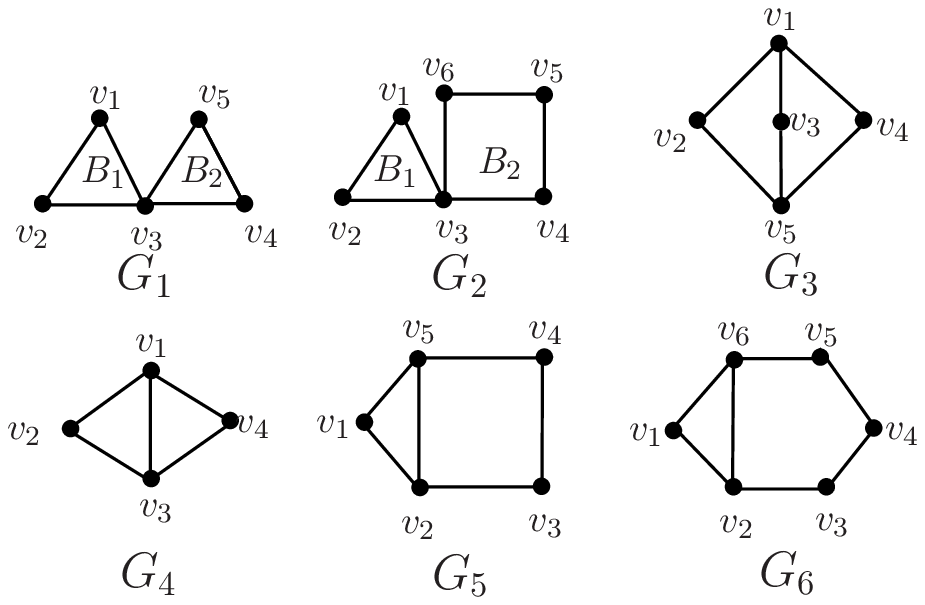}}\\[15pt]

Figure~1. The basic graphs for Lemma \ref{lem8}.
\end{center}
\end{figure}

\begin{thm}[\cite{CLYZ}]\label{thm8}
Let $G$ be a 2-edge-connected graph of order $n~(n\geq 4)$. Then $rx_3(G)\leq n-2$, with equality if and only if $G=C_n$ or $G$ is a spanning subgraph of one of the following graphs: a 3-sun, $K_5-e$, $K_4$, $G_1$, $G_2$, $H_1$, $H_2$, $H_3$, where $G_1$, $G_2$ are defined in Figure 1 and $H_1$, $H_2$, $H_3$ are defined in Figure 2.
\end{thm}

Since all the 2-edge-connected graphs with the 3-rainbow index  $n-2$ have been characterized in Theorem \ref{thm8}, it remains to characterize the graphs with 3-rainbow index $n-2$ which have cut edges. Notice that the cut edges of a graph must be assigned with distinct colors, our main purpose is to check out how the addition of cut edges to $G$ affect the 3-rainbow index of a 2-connedted graph $G$ when $rx_3(G)=n-2$. In other words, share the colors of cut edges with the colors of the non-cut edges as many as possible.

Given a connected graph $G$ of order $n$, and a coloring $c$ of $G$, we always let $A_1$ be the set of colors assigned to the non-cut edges of $G$ and $A_2$ the set of colors assigned to the cut edges of $G$. For each positive integer $k$, let $N_k=\{1,2,\cdots,k\}$. We always set that $A_2=N_s$, where $s$ is the number of cut edges of $G$. Note that, $A_1$ and $A_2$ may intersect and suppose $|A_1\cap A_2|=p$. We can interchange the colors of cut edges suitably such that $A_1\cap A_2=\{1,2,\cdots,p\}$. Set $A_1\setminus A_2=\{a_1,\cdots, a_t\}$, $t\leq m-s$ and $a_j\in N_{|c|}$.

For a connected graph $G$, a $block$ is a maximal 2-connected subgraph. In this paper, we regard $K_2$ other than a block. An $internal~ cut~ edge$ is a cut edge which is on the unique path joining some two blocks. Denote the cut edges of $G$ by $e_1=x_1y_1,\cdots,e_p=x_py_p$ and the colors of these cut edges by $1,\cdots,p$, respectively. Moreover, if $x_iy_i$ is not an internal cut edge, we always set $d(x_i,B)\leq d(y_i,B)$ where $B$ is an arbitrary block.

Let $H$ be a connected subgraph of $G$, denote by $i\in H$ if the color $i$ appears in $H$. Given a graph $G$, let $G_0$ be its basic graph. Deleting the corresponding edges of $G_0$ in $G$, we obtain a forest. Each component corresponds to a vertex $v$ in $G_0$, denoted by $T(v)$. Denote by $U(v)$ the number of leaves of $T(v)$ in $G$ and $U(G)=\sum_{v\in V(G)}U(v)$. Let $W(v)$ be the number of edges of $T(v)$ whose colors are appeared in $A_1$, that is, $W(v)=|c(T(v))\cap A_1|$.

\begin{figure}[h,t,b,p]
\begin{center}
\scalebox{1}[1]{\includegraphics{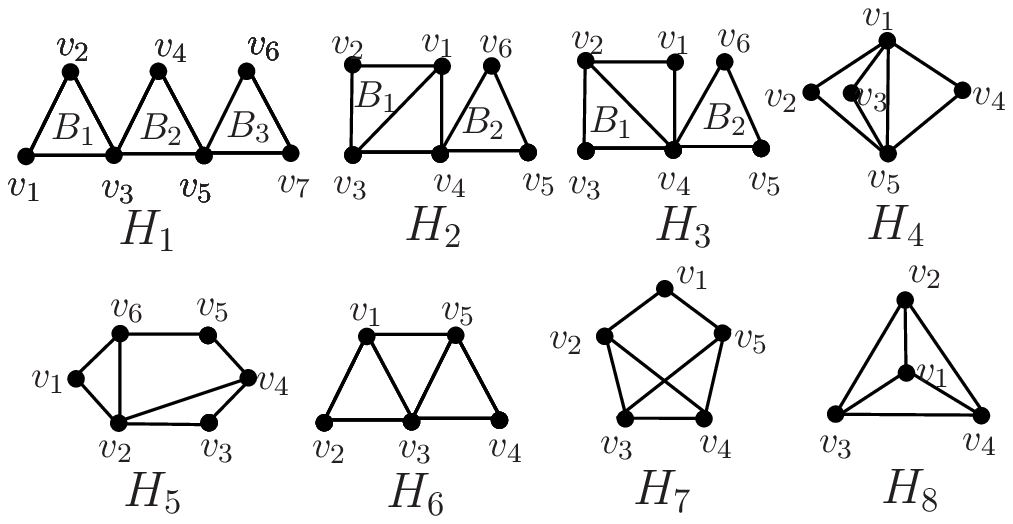}}\\[15pt]

Figure~2. The basic graphs for Lemma \ref{lem9}.
\end{center}
\end{figure}

\subsubsection{Bicyclic graphs with $rx_3(G)\bf{=n-2}$}
First, we introduce some graph classes. Let $G_i$ be the graphs shown in Figure 1, define by $\mathcal{G}^{*}_i$ the set of graphs whose basic graph is $G_i$, where $1\leq i\leq 6$. Set
$\mathcal{G}_1=\{G\in \mathcal{G}^{*}_1|U(v_3)\leq 1\}$,
$\mathcal{G}_2=\{G\in \mathcal{G}^{*}_2|U(v_3)+U(v_i)\leq 1, ~i=4,6\}$, $\mathcal{G}_3=\{G\in \mathcal{G}^{*}_3|U(v_i)+U(v_j)\leq 2,~v_iv_j\in E(G_3)\}$, $\mathcal{G}_4=\{G\in \mathcal{G}^{*}_4|U(v_i)\leq 2,~i=1,3\}$, $\mathcal{G}_5=\{G\in \mathcal{G}^{*}_5|U(v_2)+U(v_3)\leq 2,~U(v_4)+U(v_5)\leq 2\}$, $\mathcal{G}_6=\{G\in \mathcal{G}^{*}_6|U(v_2)=U(v_6)=0,~U(v_4)\leq 1,~U(v_4)+U(v_i)\leq 2,~i=3,5 \}$ and set $\mathcal{G}=\{\mathcal{G}_1,\mathcal{G}_2,\cdots,\mathcal{G}_6\}$.

\begin{lem}\label{lem8}
Let $G$ be a connected bicyclic graph of order $n$. Then $rx_3(G)=n-2$ if and only if $G\in \mathcal{G}$.
\end{lem}

\begin{proof}
Suppose that $G$ is a graph with $rx_3(G)=n-2$ but $G\notin \mathcal{G}$. Let $G_0$ be the basic graph of $G$, then $G_0$ is a 2-edge-connected bicyclic graph. If $G_0\neq G_i$, by Theorem \ref{thm8}, $rx_3(G_0)\leq |G_0|-3$. Moreover, by Lemma \ref{lem1}, we have $rx_3(G)\leq rx_3(G_0)+|G|-|G_0|\leq n-3$. Hence $G_0=G_i$. Next we show that if $G\in \mathcal{G}^{*}_i\setminus \mathcal{G}_i$, then $rx_3(G)\leq n-3$, where $1\leq i\leq 6$. As pointed out before, all the cut edges of $G$ are colored with $1,2,\cdots$. We only provide a coloring $c_{\ell}$ of $G_0$, namely, color the corresponding edges of $G$, with parts of colors used in cut edges, and the position of cut edges will be determined as following:  $\{1,2,\cdots,q\}\subseteq T(v)$ means to assign colors $\{1,2,\cdots,q\}$ to $q$ leaves of $T(v)$ arbitrarily. If $G\in \mathcal{G}^{*}_1\setminus \mathcal{G}_1$, then $U(v_3)\geq 2$, set $c_{\ell}(G_1)=1a_1a_2a_2a_12$ and $\{1,2\}\subseteq T(v_3)$. If $G\in \mathcal{G}^{*}_2\setminus \mathcal{G}_2$, then $U(v_3)+U(v_4)\geq 2$ or $U(v_3)+U(v_6)\geq 2$. By contracting $v_3v_4$ or $v_3v_6$, we obtain a graph $G'$ belonging to $\mathcal{G}^{*}_1\setminus \mathcal{G}_1$. Then the coloring of $G$ can be obtained easily from $G'$ by Lemma \ref{lem1}. If $G\in \mathcal{G}^{*}_3\setminus \mathcal{G}_3$, then there is an edge $v_iv_j\in E(G_3)$ such that $U(v_i)+U(v_j)\geq 3$. By symmetry, there exist four cases for $G$: (1) $U(v_1)\geq3$; (2) $U(v_1)\geq2$, $U(v_2)\geq1$; (3) $U(v_1)\geq1$, $U(v_2)\geq2$; (4) $U(v_2)\geq3$. Set $c_{\ell}(G_3)=a_1a_2a_2123$ and set
$\{1,2,3\}\subseteq T(v_1)$ for (1);
$\{1\}\subseteq T(v_2)$, $\{2,3\}\subseteq T(v_1)$ for (2);
$\{1\}\subseteq T(v_1)$, $\{2,3\}\subseteq T(v_2)$ for (3);
$\{1,2,3\}\subseteq T(v_2)$ for (4). If $G\in \mathcal{G}^{*}_4\setminus \mathcal{G}_4$, then $U(v_1)\geq3$ or $U(v_3)\geq3$. By symmetry, suppose $U(v_1)\geq 3$ and set $c_{\ell}(G_4)=123a_1a_1$ and $\{1,2,3\}\subseteq T(v_1)$. If $G\in \mathcal{G}^{*}_5\setminus \mathcal{G}_5$, then by contracting $v_2v_3$ or $v_4v_5$, we obtain a graph $G'$ belonging to $\mathcal{G}^{*}_4\setminus \mathcal{G}_4$. Now consider $G\in \mathcal{G}^{*}_6\setminus \mathcal{G}_6$. Then $U(v_2)\geq 1$, or $U(v_6)\geq 1$, or $U(v_4)\geq 2$, or $U(v_4)+U(v_3)\geq 3$, or $U(v_4)+U(v_5)\geq 3$. For the last two cases, it belongs to $\mathcal{G}^{*}_5\setminus \mathcal{G}_5$ by contracting $v_3v_4$ or $v_4v_5$. If $U(v_2)\geq1$, set $c_{\ell}(G_6)=a_3a_2a_4a_4a_2a_31$ and $\{1\}\subseteq T(v_2)$. If $U(v_4)\geq2$, set $c_{\ell}(G_6)=a_31a_22a_1a_2a_1$ and $\{1,2\}\subseteq T(v_4)$. It is not hard to check that the colorings above make $G$ rainbow connected with $n-3$ colors, thus $rx_3(G)\leq n-3$.

Conversely, let $G$ be a bicyclic graph such that $G\in \mathcal{G}$. Assume, to the contrary, that $rx_3(G)\leq n-3$. Then there exists a rainbow coloring $c$ such that $A_1\cup A_2=N_{n-3}$. By Theorem \ref{thm8}, we focus on the graphs with cut edges and $|A_1\cap A_2|\geq 1$. We write $d_{G_i}(u,v,w)$ to mean that the number of edges of a $\{u,v,w\}$-tree in $G$ which correspond to the edges of $G_i$, the basic of $G$. We divide into three cases.

{\bf Case 1.}~~$G\in \mathcal{G}_1\cup \mathcal{G}_2$.

First assume that $G\in \mathcal{G}_2$ and we give the following claims. If there is a nontrivial path $P_{\ell}$ connecting $B_1$ and $B_2$ in $G$, then denote its ends by $v_3'(\in B_1)$ and $v_3''(\in B_2)$.

{\bf Claim 1.}~~Each block $B_i$ has at most one edge use the color from $A_2$, where $i\in \{1,2\}$. Moreover, if a color of $A_2$ appears in $B_i$, then the other edges of $B_i$ must be assigned with different colors in $A_1\setminus A_2$.

\emph{Proof.}~Suppose two edges of $B_1$ are colored with 1, 2, respectively. We also set $d(x_i,B_1)\leq d(y_i,B_1)$, where $x_iy_i$ belongs to $P_{\ell}$. Since the cut edges colored with 1 and 2 should be contained in the rainbow tree whose vertices contain $y_1$ and $y_2$, by deleting the edges assigned with 1 and 2 in $B_1$, $G$ is disconnected. Let $w$ be a vertex in the component that does not contain $y_1$, then there is no rainbow tree connecting $\{y_1,y_2,w\}$, a contradiction. We can take the similar argument for the other cases when two edges of $B_1$ $(B_2)$ are colored with 1 or two edges of $B_2$ are colored with 1, 2, respectively.

Now suppose $1\in B_i\cap A_2$ and two edges of $B_i$ have the same color $a_1$. Let $w_1$, $w_2$ be the end vertices of the edge assigned with 1, then $\{y_1,w_1,w_2\}$ has no rainbow tree.  $\square$

{\bf Claim 2}~~The colors of the path $P_{\ell}$ can not appear in $A_1$.

\emph{Proof.}~ Assume $e$ is the edge of $P_{\ell}$ colored with 1. The color 1 can not appear in $B_1$. Otherwise suppose the three edges of $B_1$ are assigned with 1, $a_1$ and $a_2$, respectively. Consider $\{v_1,v_2,v_5\}$, then $c(v_3^{''}v_4),c(v_4v_5)\in \{2,a_3\}$ or $c(v_3^{''}v_6),c(v_5v_6)\in \{2,a_3\}$. Without loss of generality, suppose $c(v_3^{''}v_4),c(v_4v_5)\in \{2,a_3\}$, then by Claim 1, $c(v_3^{''}v_6),c(v_5v_6)\in \{a_1,a_2\}$, thus $\{v_1,v_2,v_6\}$ has no rainbow tree. On the other hand, 1 can not be in $B_2$. It is easy to see that neither $c(v_3^{''}v_4)$ nor $c(v_3^{''}v_6)$ can be 1 by considering $\{v_1,v_2,v_6\}$. If $c(v_5v_6)=1$, consider $\{v_1,v_5,v_6\}$, $\{v_2,v_5,v_6\}$, then $c(v_1v_3^{'}),c(v_2v_3^{'})\in A_2$, a contradiction to Claim 1. $\square$

By Claim 1, we have $1\leq|A_1\cap A_2|\leq2$ and only color 1 and 2 can exist in $A_1$. We should discuss all the situations according to which cut edges are colored with 1, 2. By the definition of $G$, $U(v_3)=1$ or $U(v_3)=0$. By similarity, we only deal with the former case. First assume $|A_1\cap A_2|=1$, then $A_1=\{1,a_1,a_2,a_3\}$.
We consider the subcase when $1\in T(v_3)$. In this case we claim that the color $1$ appears in neither $B_1$ nor $B_2$. Indeed, if $c(v_3''v_6)=1$, since every tree whose vertices contain $y_1$ must contain the cut edge colored with 1, $d_{G_2}(y_1,v_1,v_6)=4$. Thus $\{y_1,v_1,v_6\}$ has no rainbow tree. If now $c(v_5v_6)=1$, then consider $\{y_1,v_5,v_6\}$, $\{y_1,v_5,v_1\}$, $\{y_1,v_5,v_2\}$ successively, we have $c(v_1v_3')=c(v_2v_3')=c(v_3''v_6)$, leading to a contradiction when considering $\{v_1,v_2,v_6\}$. Else if $c(v_1v_3')=1$, then $\{y_1,v_1,v_5\}$ has no rainbow tree. The last possibility is that $c(v_1v_2)=1$, we may set $c(v_1v_3)=a_1$, $c(v_2v_3)=a_2$. Consider $\{y_1,v_1,v_4\}$, $\{y_1,v_2,v_4\}$, $\{y_1,v_1,v_6\}$, $\{y_1,v_2,v_6\}$ successively, we have $c(v_3''v_4)=c(v_3''v_6)=a_3$ and 1 can not appear in $B_2$, hence $\{v_1,v_4,v_6\}$ has no rainbow tree. The other subcases are similar.

Thus $|A_1\cap A_2|=2$, $A_1=\{1,2,a_1,a_2,a_3\}$. By Claim 1, set $1\in B_1$, $2\in B_2$,  and the other edges in each block have distinct colors. If $1,2\in T(v_3)$, assume that $d(y_1,T(v_3))>d(y_2,T(v_3))$, there always exist two vertices which come from different blocks such that there is no rainbow tree connecting them and $y_1$.
If $1\in T(v_3)$, $2\in T(v_1)$, the most difficult case is that $c(v_1v_2)=1$, $c(v_5v_6)=2$. In this case, consider $\{y_2,v_5,v_6\}$, forcing that one of $v_1v_3'$, $v_3''v_4$, $v_3''v_6$, $v_4v_5$ is colored with 1, contradicting to Claim 1. With an analogous argument, we would get a contradiction if 1, 2 are in other cut edges of $G$.

For $G\in \mathcal{G}_1$, it can be obtained by contracting an edge of a graph in $\mathcal{G}_2$. Then by Lemma \ref{lem1}, $rx_3(G)\geq n-2$.

{\bf Case 2.}~~$G\in \mathcal{G}_3$.

First note that each path from $v_1$ to $v_5$ in $G_3$ can have at most one color in $A_2$. Thus $|A_1\cap A_2|\leq3$. On the other hand, noticing that $d_{G_3}(v_2,v_3,v_4)=3> 2$, all the cases satisfying $W(v_1)=W(v_5)=0$ and $W(v_2)$, $W(v_3)$, $W(v_4)\leq 1$ are easy to get a contradiction, so we omit them here.

First assume $|A_1\cap A_2|=1$, then $A_1=\{1,a_1,a_2\}$.
If $1\in T(v_1)$, consider $\{y_1,v_2,v_3\}$, $\{y_1,v_2,v_4\}$ and $\{y_1,v_3,v_4\}$ successively, $v_1v_2$, $v_1v_3$, $v_1v_4$ must be colored with distinct colors from $A_1\setminus \{1\}$, which is impossible.

Assume now $|A_1\cap A_2|=2$, then $A_1=\{1,2,a_1,a_2\}$.
If $1,2\in T(v_1)$, then consider $\{y_1,y_2,v_5\}$, without loss of generality, set $c(v_1v_2)=a_1$, $c(v_2v_5)=a_2$. Thus $c(v_1v_3)$ can be neither 1 nor 2, otherwise there is no rainbow $\{y_1,y_2,v_3\}$-tree. On the other hand, $c(v_1v_3)$ cannot be $a_1$, otherwise $c(v_3v_5)=i$($i=1,2$), then $\{y_i,v_2,v_3\}$ has no rainbow tree. Meanwhile, $v_1v_3$ cannot be colored with $a_2$, otherwise $c(v_3v_5)=i$($i=1,2$), then $\{y_i,v_3,v_5\}$ has no rainbow tree.
If $1\in T(v_1)$, $2\in T(v_5)$, then every path from $v_1$ to $v_5$ must color $\{a_1,a_2\}$, a contradiction to $|A_1\cap A_2|=3$.
If $1,2\in T(v_2)$, then by the same reason, we conclude that $c(v_1v_2)$, $c(v_2v_5)\notin \{1,2\}$ and we may set $c(v_1v_3)=1$. But now $\{y_1,v_1,v_3\}$ has no rainbow tree.
If $1\in T(v_1)$, $2\in T(v_2)$. By considering $\{y_1,y_2,v_3\}$, $\{y_1,y_2,v_4\}$, $\{y_1,y_2,v_5\}$, we may set $c(v_1v_3)=c(v_1v_4)=c(v_2v_5)=a_1$, this force $c(v_3v_5)=i$($i=1,2$). However, there is no rainbow tree connecting $\{y_i,v_3,v_5\}$.

Thus $|A_1\cap A_2|=3$, $A_1=\{1,2,3,a_1,a_2\}$. If $1,2,3\in T(v_1)$, since $U(v_1)\leq 2$, we may assume that $y_2$ is on the unique path from $y_1$ to $v_1$. Thus one path from $v_1$ to $v_5$ must be colored with $\{a_1,a_2\}$, a contradiction to $|A_1\cap A_2|=3$.
If $1,2\in T(v_2)$, $3\in T(v_5)$, and without loss of generality, $y_2$ is on the unique path from $y_1$ to $v_2$. Considering $\{y_1,v_3,y_3\}$ and $\{y_1,v_4,y_3\}$, we may set $c(v_1v_2)=a_1$, $c(v_1v_3)=c(v_1v_4)=a_2$. But there is no rainbow $\{y_1,v_3,v_4\}$-tree. Each other case is similar or easier.

{\bf Case 3.}~~$G\in \mathcal{G}_4\cup\mathcal{G}_5\cup\mathcal{G}_6$.

First let $G\in \mathcal{G}_6$. Similarly, each path from $v_2$ to $v_6$ in $G_6$ can have at most one color in $A_2$. Thus we have $1\leq|A_1\cap A_2|\leq3$. Assume first $|A_1\cap A_2|=1$, $A_1=\{1,a_1,a_2,a_3\}$.

We only focus on the case that $1\in T(v_4)$. To make sure there are rainbow trees connecting $\{y_1,v_1,v_3\}$ and $\{y_1,v_1,v_5\}$, only $c(v_2v_6)$ can be 1, but now $\{y_1,v_2,v_6\}$ has no rainbow tree.

Assume then $|A_1\cap A_2|=2$, $A_1=\{1,2,a_1,a_2,a_3\}$. If $1,2\in T(v_1)$, we may set $c(v_1v_2)=a_1$, $c(v_2v_3)=a_2$, $c(v_3v_4)=a_3$ by considering $\{y_1,y_2,v_4\}$. $c(v_5v_6)$ can be neither 1 nor 2, otherwise $\{y_1,y_2,v_5\}$ has no rainbow tree. Moreover, $c(v_4v_5)$ can be neither 1 nor 2, otherwise when $c(v_4v_5)=i$($i=1,2$), there is no rainbow $\{y_i,v_4,v_5\}$-tree. Thus $v_1v_6$, $v_2v_6$ must use colors $\{1,2\}$, but now $\{y_1,y_2,v_6\}$ has no rainbow tree. If $1,2\in T(v_3)$, first we claim that at most one edge of the triangle $v_1v_2v_6$ uses a color from $\{1,2\}$. Otherwise if $c(v_1v_2)$, $c(v_2v_6)\in \{1,2\}$, then $\{y_1,y_2,v_1\}$ has no rainbow tree. If $c(v_1v_6)$, $c(v_2v_6)\in \{1,2\}$, the rest non-cut edges must color $\{a_1,a_2,a_3\}$. It is easy to verify that either $\{y_1,v_1,v_5\}$ or $\{y_1,v_1,v_4\}$ has no rainbow tree. So the longest path from $v_2$ to $v_6$ has an edge colored with 1 or 2. However, we will show that it is impossible. It is easy to check that $c(v_2v_3),c(v_3v_4)\notin \{1,2\}$. If $c(v_5v_6)\in \{1,2\}$, then we may set $c(v_5v_6)=1$. Consider $\{y_1,v_1,v_5\}$ and $\{y_1,v_5,v_6\}$, then $c(v_1v_2)=c(v_2v_6)=2$, a contradiction. It is similar to check that $c(v_4v_5)$ can not be 1 or 2, a contradiction.

Now assume that $|A_1\cap A_2|=3$, $A_1=\{1,2,3,a_1,a_2,a_3\}$.
If $1\in T(v_1)$, $2,3\in T(v_3)$. Again, we may set $c(v_1v_2)=a_1$, $c(v_2v_3)=a_2$. Thus $c(v_1v_6),c(v_2v_6)\in \{1,2,3\}$. If $c(v_1v_6),c(v_2v_6)\in \{1,i\}$, then there is a contradiction by considering $\{y_1,y_i,v_6\}$ ($i=2,3$). Thus we may set that $c(v_1v_6)=2$, $c(v_2v_6)=3$. By considering $\{y_1,y_3,v_4\}$ and $\{y_1,y_3,v_5\}$, we get that $c(v_3v_4)=c(v_5v_6)=a_3$, $c(v_4v_5)=1$, but now $\{y_1,v_4,v_5\}$ has no rainbow tree.
If $1\in T(v_3)$, $2,3\in T(v_5)$, then we set $v_3v_4=a_1$, $v_4v_5=a_2$. If $c(v_2v_6)=i$, $c(v_5v_6)=j$, $i,j\in\{1,2,3\}$, then $\{y_i,y_j,v_6\}$ has no rainbow tree. The only possibility is $c(v_2v_3)=2$, $c(v_2v_6)=3$, $c(v_5v_6)=a_3$. However, $\{y_1,y_2,v_1\}$ has no rainbow tree.

For $G\in \mathcal{G}_5$, notice that $|A_1\cap A_2|\leq3$. If $U(v_2)=0$ or $U(v_5)=0$, then $G$ can be obtained by contracting an edge of a graph in $\mathcal{G}_6$. Then by Lemma \ref{lem1}, $rx_3(G)\geq n-2$. Thus we need to consider the case when $W(v_2)\geq 1$ and $W(v_5)\geq 1$. If $|A_1\cap A_2|=2$, then suppose $1\in T(v_2)$, $2\in T(v_5)$. Consider $\{y_1,y_2,v_3\}$, $\{y_1,y_2,v_4\}$, we have $c(v_2v_3)$, $c(v_2v_5)$, $c(v_4v_5)\in \{a_1,a_2\}$ and $c(v_2v_3)=c(v_4v_5)$. But now $c(v_3v_4)=i$ ($i=1,2$), then there is no rainbow $\{y_i,v_3,v_4\}$-tree. If $|A_1\cap A_2|=3$, then $A_1=\{1,2,3,a_1,a_2\}$. If $1\in T(v_1)$, $2\in T(v_2)$, $3\in T(v_5)$, then consider $\{y_1,y_2,y_3\}$, we have that two of $v_1v_2$, $v_1v_5$, $v_2v_5$ have colors outside $A_2$, contradicting to $|A_1\cap A_2|=3$. If $1\in T(v_2)$, $2\in T(v_5)$, $3\in T(v_2)$ and we may assume that $y_3$ is on the unique path from $y_1$ to $v_2$. Then consider $\{y_1,v_3,y_3\}$ and $\{y_1,v_4,y_3\}$, we have $c(v_2v_3)=c(v_4v_5)$, thus $c(v_3v_4)$ can not be in $A_2$, contradicting to $|A_1\cap A_2|=3$. If $1\in T(v_2)$, $2\in T(v_5)$, $3\in T(v_i)$ ($i=3,4$), then consider $\{y_1,y_2,y_3\}$, we have that $c(v_2v_5)$ is in $A_1\setminus A_2$, contradicting to $|A_1\cap A_2|=3$.

Finally, for a graph $G$ belonging to $\mathcal{G}_4$, it can be obtained by contracting an edge of a graph in $\mathcal{G}_3\cup \mathcal{G}_6$. Then by Lemma \ref{lem1}, $rx_3(G)\geq n-2$.

Combining all the cases above, we have $rx_3(G)\geq n-2$ for $G\in \mathcal{G}$. By Theorem \ref{thm4}, it follows that $rx_3(G)=n-2$.
\end{proof}

\subsubsection{Tricyclic graphs with $rx_3(G)\bf{=n-2}$}
Define by $\mathcal{H}^{*}_i$ the set of graphs whose basic graph is $H_i$, where  $H_i$ is shown in Figure 2 and $1\leq i\leq 8$.

Now, we introduce another graph class $\mathcal{H}$. Set $\mathcal{H}=\{\mathcal{H}_1,\mathcal{H}_2,\cdots,\mathcal{H}_8\}$, where
$\mathcal{H}_1=\{G\in \mathcal{H}^{*}_1|U(G)=0\}$,
$\mathcal{H}_2=\{G\in \mathcal{H}^{*}_2|U(v_i)\leq 1, ~U(v_j)=0,~i=5,6,~j=1,3,4\}$,
$\mathcal{H}_3=\{G\in \mathcal{H}^{*}_3|U(v_2)\leq 1,~ U(v_5)+U(v_6)\leq 1,~ U(v_i)=0,~i=1,3,4\}$, $\mathcal{H}_4=\{G\in \mathcal{H}^{*}_4|U(v_i)\leq 1,~ U(v_j)\leq 2, ~U(v_i)+U(v_j)\leq 1,~U(v_j)+U(v_k)\leq 3,~ i=1,5,~ j,k=2,3,4\}$,
$\mathcal{H}_5=\{G\in \mathcal{H}^{*}_5|U(v_i)\leq 1,~U(v_j)=0,~i=1,3,5,~j=2,4,6\}$, $\mathcal{H}_6=\{G\in \mathcal{H}^{*}_6|U(v_3)=0,~ U(v_i)\leq 1, ~ U(v_1)+U(v_5)\leq 1,~ i=1,2,4,5\}$,
$\mathcal{H}_7=\{G\in \mathcal{H}^{*}_2|U(v_2)+U(v_4)\leq 1, ~U(v_3)+U(v_5)\leq 1,~ U(v_5)+U(v_1)\leq 1,~ U(v_j)+U(v_{j+1})\leq 1,~ j=1,2,4\}$,
$\mathcal{H}_8=\{G\in \mathcal{H}^{*}_8|U(v_i)\leq 2,~U(v_i)+U(v_j)+U(v_k)\leq 3, ~i,j,k=1,2,3,4\}$.

\begin{lem}\label{lem9}
Let $G$ be a connected tricyclic graph of order $n$. Then $rx_3(G)=n-2$ if and only if $G\in \mathcal{H}$.
\end{lem}
\begin{proof}
Suppose that $rx_3(G)=n-2$ but $G\notin \mathcal{H}$. Let $G_0$ be the basic graph of $G$. Similar to Lemma \ref{lem8}, we have $G_0=H_i$ and we can rainbow color $G$ with $n-3$ colors for $G\in \mathcal{H}_i^{*}\setminus \mathcal{H}_i$, $i=1,\cdots,8$.

If $G\in \mathcal{H}_1^{*}\setminus \mathcal{H}_1$, then if $U(v_2)\geq 1$, set $c_{\ell}(H_1)=a_41a_1a_2a_2a_4a_3a_3a_4$ and $\{1\}\subseteq T(v_2)$; if $U(v_3)\geq 1$, set $c_{\ell}(H_1)=a_4a_1a_11a_3a_4a_2a_2a_4$ and $\{1\}\subseteq T(v_3)$; if $U(v_4)\geq 1$, set $c_{\ell}(H_1)=a_3a_2a_2a_11a_3a_4a_4a_1$ and $\{1\}\subseteq T(v_4)$.

If $G\in \mathcal{H}_2^{*}\setminus \mathcal{H}_2$, then if $U(v_3)\geq 1$ ($U(v_1)\geq 1$ is similar), set $c_{\ell}(H_2)=a_3a_21a_2a_3a_1a_1a_2$ and $\{1\}\subseteq T(v_3)$; if $U(v_4)\geq 1$, set $c_{\ell}(H_2)=a_1a_3a_21a_2a_3a_3a_1$ and $\{1\}\subseteq T(v_4)$; if $U(v_6)\geq 2$ ($U(v_5)\geq 2$ is similar), set $c_{\ell}(H_2)=a_31a_22a_2a_3a_11$ and $\{1,2\}\subseteq T(v_3)$.

If $G\in \mathcal{H}_3^{*}\setminus \mathcal{H}_3$, then if $U(v_2)\geq 2$, set $c_{\ell}(H_3)=a_12a_2a_32a_2a_11$ and $\{1,2\}\subseteq T(v_2)$; if $U(v_5)\geq 2$, set $c_{\ell}(H_3)=2a_2a_31a_2a_11a_3$ and $\{1,2\}\subseteq T(v_5)$; if $U(v_5)+U(v_6)\geq 2$, set $c_{\ell}(H_3)=1a_12a_2a_3a_1a_2a_3$ and $\{1\}\subseteq T(v_6)$, $\{2\}\subseteq T(v_5)$; if $U(v_3)\geq 1$, set $c_{\ell}(H_3)=a_1a_2a_21a_1a_3a_3a_2$ and $\{1\}\subseteq T(v_3)$; if $U(v_4)\geq 1$, set $c_{\ell}(H_3)=a_2a_11a_1a_2a_3a_3a_2$ and $\{1\}\subseteq T(v_4)$.

If $G\in \mathcal{H}_4^{*}\setminus \mathcal{H}_4$, then there are four cases for the graph $G$: (1) $U(v_i)\geq 3$ ($i=2,3,4$); (2) $U(v_i)\geq 2$ ($i=1,5$); (3) $U(v_i)+U(v_j)\geq 2$ ($i\in\{1,5\}$, $j\in\{2,3,4\}$); (4) $U(v_i)\geq 2$ and $U(v_j)\geq 2$, $i,j\in\{2,3,4\}$. If $G$ is a graph in case (1), then there exists a graph in $\mathcal{G}_3^{*}\setminus \mathcal{G}_3$ which is a subgraph of $G$. Thus the result is obvious. If $U(v_1)\geq 2$, set $c_{\ell}(H_4)=a_2a_2a_1a_121a_2$ and $\{1,2\}\subseteq T(v_1)$; if $U(v_1)+U(v_2)\geq 2$, set $c_{\ell}(H_4)=a_1a_2a_22a_21a_1$ and $\{1\}\subseteq T(v_1)$, $\{2\}\subseteq T(v_2)$; if $U(v_2)\geq 2$ and $U(v_4)\geq 2$, set $c_{\ell}(H_4)=3421a_2a_2a_1$ and $\{1,2\}\subseteq T(v_2)$, $\{3,4\}\subseteq T(v_4)$.

If $G\in \mathcal{H}_5^{*}\setminus \mathcal{H}_5$, then $U(v_i)\geq 1$ ($i=2,4,6$) or $U(v_i)\geq 2$ ($i=1,3,5$). If $G$ is a graph in the former case, there exists a graph in $\mathcal{G}_6^{*}\setminus \mathcal{G}_6$ which is a subgraph of $G$.
If $U(v_3)\geq 2$, set $c_{\ell}(H_5)=a_22a_1a_2a_31a_3a_2$ and $\{1,2\}\subseteq T(v_3)$; if $U(v_5)\geq 2$, set $c_{\ell}(H_5)=1a_22a_1a_3a_2a_3a_1$ and $\{1,2\}\subseteq T(v_5)$;

If $G\in \mathcal{H}_6^{*}\setminus \mathcal{H}_6$, then $U(v_3)\geq 1$ or $U(v_i)\geq 2$ ($i=1,2,4,5$) or $U(v_1)+U(v_5)\geq 2$. If $U(v_3)\geq 1$, set $c_{\ell}(H_6)=a_2a_11a_1a_2a_2a_1$ and $\{1\}\subseteq T(v_3)$;
if $U(v_1)\geq 2$, set $c_{\ell}(H_6)=a_2a_1a_11a_212$ and $\{1,2\}\subseteq T(v_1)$; if $U(v_2)\geq 2$, set $c_{\ell}(H_6)=a_21a_1a_1a_212$ and $\{1,2\}\subseteq T(v_2)$;
if $U(v_1)+U(v_5)\geq 2$, set $c_{\ell}(H_6)=a_1a_1a_2a_221a_1$ and $\{1\}\subseteq T(v_1)$, $\{2\}\subseteq T(v_5)$;

If $G\in \mathcal{H}_7^{*}\setminus \mathcal{H}_7$, then $U(v_i)\geq 2$ ($i=1,2,3,4,5$) or $U(v_1)+U(v_2)\geq 2$ or $U(v_2)+U(v_3)\geq 2$ or $U(v_4)+U(v_5)\geq 2$ or $U(v_2)+U(v_4)\geq 2$ or $U(v_3)+U(v_5)\geq 2$ or $U(v_1)+U(v_5)\geq 2$. If $U(v_1)\geq 2$, set $c_{\ell}(H_7)=a_1a_2a_212a_1a_1$ and $\{1,2\}\subseteq T(v_3)$;
if $U(v_2)\geq 2$, set $c_{\ell}(H_7)=a_2a_1a_1a_1a_212$ and $\{1,2\}\subseteq T(v_2)$; if $U(v_3)\geq 2$, set $c_{\ell}(H_7)=a_2a_1a_12a_1a_21$ and $\{1,2\}\subseteq T(v_3)$;
if $U(v_1)+U(v_2)\geq 2$, set $c_{\ell}(H_7)=a_2a_1a_1a_1a_221$ and $\{1\}\subseteq T(v_1)$, $\{2\}\subseteq T(v_2)$; if $U(v_2)+U(v_3)\geq 2$, set $c_{\ell}(H_7)=a_2a_1a_12a_2a_21$ and $\{1\}\subseteq T(v_2)$, $\{2\}\subseteq T(v_3)$; if $U(v_2)+U(v_4)\geq 2$, set $c_{\ell}(H_7)=a_212a_1a_2a_1a_2$ and $\{1\}\subseteq T(v_2)$, $\{2\}\subseteq T(v_4)$;

If $G\in \mathcal{H}_8^{*}\setminus \mathcal{H}_8$, then $U(v_i)\geq 3$ ($i=1,2,3,4$) or $U(v_i)+U(v_j)+U(v_k)\geq 4, ~i,j,k=1,2,3,4$. If $G$ is a graph in the former case, then a graph belonging to $\mathcal{G}_4^{*}\setminus \mathcal{G}_4$ is a subgraph of $G$. If $U(v_1)+U(v_2)+U(v_4)\geq 4$, set $c_{\ell}(H_8)=1a_1a_1423$ and $\{1,2\}\subseteq T(v_1)$, $\{3\}\subseteq T(v_2)$, $\{4\}\subseteq T(v_4)$; if $U(v_2)+U(v_3)\geq 4$, set $c_{\ell}(H_8)=12a_1a_134$ and $\{1,2\}\subseteq T(v_2)$, $\{3,4\}\subseteq T(v_3)$.

It is not hard to check that the colorings above make $G$ rainbow connected with $n-3$ colors, thus $rx_3(G)\leq n-3$.

Conversely, let $G$ be a tricyclic graph such that $G\in \mathcal{H}$. Similar to Lemma \ref{lem8}, we only need to consider the case that $G$ has cut edges and $|A_1\cap A_2|\geq 1$. Assume, to the contrary, that $rx_3(G)\leq n-3$. Then there exists a rainbow coloring $c$ of $G$ using colors in $N_{n-3}$.

For $G\in \mathcal{H}_1$, if there is a nontrivial path $P'$ connecting $B_1$ and $B_2$ in $G$, then denote its ends by $v_3'(\in B_1)$ and $v_3''(\in B_2)$ and if there is a nontrivial path $P''$ connecting $B_2$ and $B_3$ in $G$, then denote its ends by $v_5'(\in B_2)$ and $v_5''(\in B_3)$. Similar to Claim 2 in Lemma \ref{lem8}, the colors in the path $P^{'}$ and $P^{''}$ can not appear in $A_1$, which implies $|A_1\cap A_2|=0$, contradicting to $|A_1\cap A_2|\geq 1$. For $G\in \mathcal{H}_5$, notice that $d_{H_5}(v_1,v_3,v_5)= 4$ and $|A_1\setminus A_2|=3$, the result holds. The same argument applies to the case when $G\in \mathcal{H}_6$. Thus, we mainly discuss the rest cases for $G$ as follows.

{\bf Case 1.}~~$G\in \mathcal{H}_2$. we have $1\leq|A_1\cap A_2|\leq4$ and $|A_1\setminus A_2|=3$. If there is a nontrivial path $P'$ connecting $B_1$ and $B_2$ in $G$, then denote its ends by $v_4'(\in B_1)$ and $v_4''(\in B_2)$. We can also claim that $c(P')\cap A_1= \emptyset$. Noticing that $d_{H_2}(v_2,v_5,v_6)=4>3$, we only check the case when $W(v_2)\geq 2$.  Since the case of $|A_1\cap A_2|=1$ or $|A_1\cap A_2|=4$ is easy to check, we consider the remaining two cases. Assume $|A_1\cap A_2|=2$, $A_1=\{1,2,a_1,a_2,a_3\}$ and $1,2\in T(v_2)$, consider $\{y_1,y_2,v_5\}$ and we may set $c(v_2v_3)=a_1$, $c(v_3v_4')=a_2$, $c(v_4''v_5)=a_3$. If 1 and 2 are in $B_1$, and 1 appears in $v_1v_2$ or $v_1v_4'$, then we have $c(v_4''v_6)=a_3$ by considering $\{y_1,y_2,v_6\}$, and thus $c(v_5v_6)\notin A_2$, but now $\{y_1,v_5,v_6\}$ has no rainbow tree. So one of 1, 2, say 1, is in $B_2$ and $c(v_5v_6)=1$. Now we have $c(v_4''v_6)\neq a_3$ and $c(v_1v_2),c(v_1v_4'),c(v_4''v_6)\in \{a_1,a_2,a_3\}$ by considering $\{y_1,y_2,v_6\}$. Then every $\{y_1,v_5,v_6\}$-tree of size 5 can not have the color 2. Thus there is no rainbow $\{y_1,v_5,v_6\}$-tree.

Assume then $|A_1\cap A_2|=3$ and $A_1=\{1,2,a_1,a_2,a_3\}$.
If $1,2,3\in T(v_2)$, first we claim that $v_2v_3,v_3v_4'$ can not use colors from $A_2$ both. Otherwise assume $c(v_2v_3)=1$, $c(v_3v_4)=2$, $c(v_1v_2)=a_1$, $c(v_1v_4')=a_2$, and by considering $\{y_1,y_2,v_5\}$ and $\{y_1,y_2,v_6\}$, we have $c(v_4''v_5)=c(v_4''v_6)=a_3$, and $c(v_5v_6)$ can be $a_1$ or $a_2$. However, there is no rainbow $\{y_1,v_5,v_6\}$-tree or $\{y_2,v_5,v_6\}$-tree. With the same reason, we conclude that exactly one edge of the unique 4-cycle of $H_2$ can be colored with a color from $A_2$.  Thus there are two cases by symmetry. If $c(v_1v_2)=1$, $c(v_1v_3)=2$, $c(v_5v_6)=3$. Consider $y_1$ and $y_2$, together with $v_5$, $v_6$ respectively, we have $c(v_4''v_5)=c(v_4''v_6)$, which is impossible. If $c(v_1v_4')=1$, $c(v_1v_3)=2$, $c(v_5v_6)=3$. Consider $y_1$ and $y_3$, together with $v_5$, $v_6$ respectively. suppose $c(v_4''v_5)=a_1$, $c(v_4''v_6)=a_2$ and $c(v_3v_4')=a_3$, $c(v_1v_2),c(v_2v_3)\in \{a_1,a_2\}$, but now there is no rainbow $\{y_3,v_5,v_6\}$-tree.
If $1,2\in T(v_2)$, $3\in T(v_6)$, similarly set $c(v_1v_2)=a_1$, $c(v_1v_4')=a_2$, $c(v_4''v_6)=a_3$. First we can easily claim that the color 3 can not appear in $B_2$. Thus there are three possibilities for the color 3. If $c(v_3v_4')=3$, then consider $\{y_1,y_3,v_3\}$ and $\{y_2,y_3,v_3\}$, we have $c(v_1v_3),c(v_2v_3)\in \{1,2\}$. Consider $\{y_1,y_2,v_5\}$, one of $c(v_4''v_5)$ and $c(v_5v_6)$ is $a_3$, but now there is no rainbow $\{y_2,y_3,v_5\}$-tree. The case when  $c(v_1v_3)=3$ is similar to the case of $c(v_3v_4')=3$. If $c(v_2v_3)=3$, then similarly we get $c(v_1v_3),c(v_2v_3)\in \{1,2\}$ and one of $c(v_4''v_5)$ and $c(v_5v_6)$ is $a_3$. Consider $\{y_1,y_3,v_5\}$, this forces one of $c(v_4''v_5)$ and $c(v_5v_6)$ is $2$, it is impossible.

{\bf Case 2.}~~$G\in \mathcal{H}_3$. Then $1\leq|A_1\cap A_2|\leq 4$ and $|A_1\setminus A_2|=3$. If there is a nontrivial path $P'$ connecting $B_1$ and $B_2$ in $G$, then denote its ends by $v_4'(\in B_1)$ and $v_4''(\in B_2)$. Similarly, it is easy to check that $c(P^{'})\cap A_1= \emptyset$. We only focus on the case that $|A_1\cap A_2|=1$, where $A_1=\{1,a_1,a_2,a_3\}$. If $1\in T(v_6)$, consider $\{y_1,v_1,v_3\}$, set $c(v_4''v_6)=a_1$, $c(v_1v_4')=a_2$, $c(v_3v_4')=a_3$. Then $c(v_5v_6)\neq 1$, otherwise there is no rainbow tree connecting $\{y_1,v_1,v_5\}$ or $\{y_1,v_3,v_5\}$ depending on $c(v_4''v_5)$. Similarly, $c(v_4''v_5)\neq 1$. Next $c(v_2v_4')\neq 1$ by considering $\{y_1,v_2,v_5\}$. Suppose $c(v_1v_2)=1$, to make sure there is a rainbow $\{y_1,v_1,v_2\}$-tree and $\{y_1,v_2,v_3\}$-tree, we have $c(v_2v_4')=a_3$ and $c(v_2v_3)=a_2$. But now $\{v_2,v_3,v_5\}$ has no rainbow tree. If $1\in T(v_2)$, then consider $\{y_1,v_5,v_6\}$, $\{y_1,v_1,v_5\}$, $\{y_1,v_1,v_6\}$, $\{y_1,v_3,v_5\}$, $\{y_1,v_3,v_6\}$ successively. Set $c(v_2v_4')=a_1$, then $c(v_1v_2)$, $c(v_1v_4')$, $c(v_2v_3)$, $c(v_3v_4')$, $c(v_4''v_5)$ and $c(v_4''v_6)$ can only be $a_2$ or $a_3$. It is easy to check that $\{v_2,v_3,v_5\}$ has no rainbow tree.

{\bf Case 3.}~~$G\in \mathcal{H}_4$. $1\leq |A_1\cap A_2|\leq 4$ and $|A_1\setminus A_2|=2$. First notice that  $d_{H_4}(v_2,v_3,v_4)=3$, the case that $W(v_1)=W(v_5)=0$, $W(v_2)$, $W(v_3)$, $W(v_4)\leq1$ is evident.
Assume $|A_1\cap A_2|=1$, then $1\in T(v_1)$, the case is similar with the case that $G\in \mathcal{G}_3$ in Lemma \ref{lem8}. Assume now $|A_1\cap A_2|=2$, if $1\in T(v_1)$, $2\in T(v_5)$, by considering all the trees containing $y_1$ and $y_2$, without loss of generality, set $c(v_1v_5)=a_1$, $c(v_1v_2)=1(2)$, $c(v_2v_5)=a_2$. Moreover, by considering $\{y_1(y_2),v_2,v_3\}$ and $\{y_1(y_2),v_2,v_4\}$, the remaining two paths of length 2 from $v_1$ to $v_5$ must be colored with 2(1), $a_2$, respectively. However, there is no rainbow $\{y_2(y_1),v_3,v_4\}$-tree.
If $1,2\in T(v_2)$, by considering $\{y_1,y_2,v_4\}$, set  $c(v_2v_5)=a_1$, $c(v_4v_5)=a_2$. Since the two possible rainbow trees connecting $\{y_1,v_3,v_4\}$ and $\{y_2,v_3,v_4\}$ are the same, we may set $c(v_3v_5)=1$. It is easy to see that $c(v_1v_2)$, $c(v_1v_3)$ cannot use colors from $A_2$ by considering $\{y_1,y_2,v_3\}$, and $c(v_1v_4)=2$ by considering $\{y_1,v_3,v_4\}$. But now if $c(v_1v_5)=1$ or $c(v_1v_5)=2$, there is no rainbow $\{y_1,v_2,v_5\}$-tree or $\{y_2,v_1,v_4\}$-tree, respectively.

Assume $|A_1\cap A_2|=3$, then $1,2\in T(v_2)$, $3\in T(v_4)$. Similarly as above, we may set $c(v_2v_5)=a_1$, $c(v_4v_5)=a_2$, $c(v_3v_5)=1$, $c(v_1v_5)=3$, one of $c(v_1v_2)$, $c(v_1v_4)$ is 2. However, there is no rainbow $\{y_2,y_3,v_1\}$-tree.

Finally assume $|A_1\cap A_2|=4$ and $1,2\in T(v_2)$, $3\in T(v_3)$, $4\in T(v_4)$, consider $\{y_1,y_3,y_4\}$ and $\{y_2,y_3,y_4\}$, at least four of the non-cut edges must be colored with $\{a_1,a_2\}$. This contradicts to $|A_1\cap A_2|=4$.

{\bf Case 4.}~~$G\in \mathcal{H}_7$. Since $d_{H_7}(v_1,v_3,v_4)=3$, we only focus on the case $1\in T(v_2)$. Consider all the three vertices containing $y_1$, it is not hard to obtain a contradiction.

{\bf Case 5.}~~$G\in \mathcal{H}_8$. First notice $d_{H_8}(v_1,v_2,v_3)=2$, the case that $W(v_1)$, $W(v_2)$, $W(v_3)$, $W(v_4)\leq1$ is evident. Assume $|A_1\cap A_2|=2$, then $1,2\in T(v_1)$. Consider $\{y_1,y_2,v_2\}$, $\{y_1,y_2,v_3\}$, $\{y_1,y_2,v_4\}$ successively, we have $c(v_1v_2)=c(v_1v_3)=c(v_1v_4)=a_1$. However, there is no rainbow tree connecting $\{y_1,v_2,v_3\}$ or $\{y_2,v_2,v_3\}$, a contradiction. Now focus on $|A_1\cap A_2|=3$, then $1,2\in T(v_1)$, $3\in T(v_2)$. Consider $\{y_1,y_3,v_3\}$, $\{y_1,y_3,v_4\}$, $\{y_2,y_3,v_3\}$ and $\{y_2,y_3,v_4\}$ successively, $c(v_1v_3)$, $c(v_1v_4)$, $c(v_2v_3)$, $c(v_2v_4)$ must be 1 or 2. Again, there is no rainbow $\{y_1,y_2,v_3\}$-tree.

By the detailed analysis above, we have $rx_3(G)\geq n-2$ for $G\in \mathcal{H}$. By Theorem \ref{thm4}, it follows that $rx_3(G)=n-2$.
\end{proof}

\subsubsection{Characterize the graphs with $rx_3(G)\bf{=n-2}$}
We begin with a lemma about a connected 5-cyclic graph.

\begin{lem}\label{lem11}
Let $G$ be a connected 5-cyclic graph of order $n$. Then $rx_3(G)=n-2$ if and only if $G=K_5-e$.
\end{lem}

\begin{proof}
Let $G\neq K_5-e$ and $rx_3(G)=n-2$, by Lemma \ref{lem1} and Theorem \ref{thm8}, $rx_3(G)\leq n-3$, a contradiction. Conversely, suppose $G=K_5-e$, by Theorem \ref{thm3}, $rx_3(G)\geq3$, on the other hand, $rx_3(G)\leq rx_3(C_5)=3$. Thus $rx_3(G)=n-2$.

\end{proof}

For $n\geq 3$, the $wheel$ $W_n$ is a graph constructed by joining a vertex $v_0$ to every vertex of a cycle $C_{n}:v_1,v_2,\cdots,v_{n},v_{n+1}=v_1$.

A third graph class is defined as follows. Let $\mathcal{J}_1$ be a class of graphs such that every graph  is obtained from a graph in $\mathcal{H}_5$ by adding an edge $v_4v_6$. Let $\mathcal{J}_2$ be a class of graphs such that every graph is obtained from a graph in $\mathcal{H}_7$ where $U(v_2)=0$ and $U(v_5)=0$ by adding an edge $v_2v_5$. Set $\mathcal{J}=\{\mathcal{J}_1,\mathcal{J}_2, W_4\}$.

Now we are ready to show our second main theorem of this paper.

\begin{thm}\label{thm2}
Let $G$ be a connected graph of order $n$ $(n\geq 6)$. Then $rx_3(G)=n-2$ if and only if $G$ is unicyclic with the girth of $G$ at least 4 or  $G\in \mathcal{G}\cup \mathcal{H}\cup \mathcal{J}$ or $G=K_5-e$.
\end{thm}

\begin{proof}
Let $G$ be a $t$-cyclic graph with $rx_3(G)=n-2$, but not a graph listed in the theorem. By Proposition \ref{pro1}, Theorem \ref{thm1}, Lemma \ref{lem8} and Lemma \ref{lem9}, we need to consider the cases $t\geq 4$. If $t=4$, by Theorem \ref{thm8}, the basic graph of $G$ should be a 3-sun or the basic graph of $\mathcal{J}_2$ or $W_4$. If $G\notin \mathcal{J}_1$ or $G\notin \mathcal{J}_2$, then by the similar arguments with Lemma \ref{lem8}, we have  $rx_3(G)\leq n-3$, a contradiction. If the basic graph of $G$ is $W_4$ and there are some cut edges in $G$. If $U(v_0)\geq1$, then a graph belonging to $\mathcal{G}_6^{*}\setminus \mathcal{G}_6$ and satisfying $U(v_3)\geq1$ is a subgraph of $G$. If $U(v_1)\geq1$(other cases are similar), then set $c_{\ell}(W_4)=a_21a_1a_1a_1a_2a_2a_1$ and $\{1\}\subseteq T(v_1)$. If $t\geq 5$, by Theorem \ref{thm8}, the basic graph of $G$ should be $K_5-e$, since $n\geq 6$, by the similar argument with $t=4$, we have $rx_3(G)\leq n-3$, a contradiction.

Conversely, by Theorem \ref{thm1}, Theorem \ref{thm3}, Lemma \ref{lem8}, Lemma \ref{lem9} and Lemma \ref{lem11}, suppose $G$ is a graph such that $G\in \mathcal{J}_1$ or $G\in \mathcal{J}_2$. Assume, to the contrary, that $rx_3(G)\leq n-3$. Then there exists a rainbow coloring $c$ of $G$ using $n-3$ colors. Both cases can be considered similar to the case that $G\in \mathcal{H}_5$ or $G\in \mathcal{H}_7$ in Lemma \ref{lem9}, a contradiction.
\end{proof}

\end{document}